\documentclass[11pt,a4paper]{article}
\usepackage{amsthm,amsmath,amssymb}
\usepackage{tikz,xcolor,graphicx}
\usepackage{mathtools}
\usepackage{enumitem}
\usepackage{booktabs}
\usepackage{url}
\usepackage[T1]{fontenc}
\usepackage{float}
\usepackage{indentfirst}
\usepackage{caption}

\usepackage{booktabs}

\usetikzlibrary{calc}
\usetikzlibrary{positioning}

\usepackage{graphicx}

\let\emph\relax 
\DeclareTextFontCommand{\emph}{\color{blue}\bfseries\em}

\usepackage[noabbrev,capitalize]{cleveref}
\crefname{equation}{}{}

\newtheorem{theorem}{Theorem}[section]

\newtheorem{lemma}[theorem]{Lemma}
\newtheorem{corollary}[theorem]{Corollary}

\theoremstyle{definition}
\newtheorem{definition}[theorem]{Definition}

\newtheorem{question}[theorem]{Question}

\newcommand{\cB}{\mathcal{B}}

\newcommand{\cF}{\mathcal{F}}

\newcommand{\cH}{\mathcal{H}}

\newcommand{\cL}{\mathcal{L}}
\newcommand{\cM}{\mathcal{M}}

\newcommand{\brac}[1]{\left[ #1 \right]}
\newcommand{\set}[1]{\left\{ #1 \right\}}

\newcommand{\ceil}[1]{\left\lceil #1 \right\rceil}

\newcommand{\wt}{\widetilde}
\newcommand{\wh}{\widehat}
\newcommand{\ol}{\overline}

\makeatletter
\newcommand*\bigcdot{\mathpalette\bigcdot@{.5}}
\newcommand*\bigcdot@[2]{\mathbin{\vcenter{\hbox{\scalebox{#2}{$\m@th#1\bullet$}}}}}
\makeatother

\DeclareMathOperator{\ex}{ex}
\DeclareMathOperator{\EX}{EX}

\DeclareMathOperator{\image}{Im}
\begin{document}
\title{The Tur\'an number for the edge blow-up of trees: the missing case}

\author{
	Cheng Chi\thanks{School of Mathematical Sciences, East China Normal University, 500 Dongchuan Road, Shanghai 200240, China.
		Email: 52215500038@stu.ecnu.edu.cn.}\qquad  Long-Tu Yuan\thanks{School of Mathematical Sciences and Shanghai Key Laboratory of PMMP, East China Normal University, 500 Dongchuan Road, Shanghai 200240, P.R.  China.
Email: ltyuan@math.ecnu.edu.cn. Supported in part by National Natural Science Foundation of China grant 11901554 and Science and Technology Commission
of Shanghai Municipality (No. 18dz2271000).}
}

\date{}
\maketitle
\begin{abstract}
The edge blow-up of a graph is the graph obtained from replacing each edge of it by a clique of the same size where the new vertices of the cliques are all different.
Wang, Hou, Liu and Ma determined the Tur\'{a}n number of the edge blow-up of trees except one particular case.
Answering an problem posed by them, we determined the Tur\'{a}n number of this particular case.
\end{abstract}
\section{Introduction}

Given a family of graphs $\cH$, a graph $G$ is said to be $\cH$-free ($H$-free if $\cH=\set{H}$) if $G$ does not contain any copy of $H\in\cH$ as a subgraph.
A typical problem in extremal combinatorics is the following Tur\'an-type problem: what is the maximum number of edges in an $\cH$-free graph on $n$ vertices?
The aforementioned number is called the {\it extremal number} for $\cH$ and denoted by $\ex(n,\cH)$.
Denote by $\EX(n,\cH)$ the set of $\cH$-free graphs on $n$ vertices with $\ex(n,\cH)$ edges and call a graph in $\EX(n,\cH)$ an {\it extremal graph} for $\cH$.
We use $\ex(n,H)$ and $\EX(n,H)$ instead of $\ex(n,\cH)$ and $\EX(n,\cH)$ respectively when $\cH=\set{H}$.
Much interests has been attracted to this problem during the last few decades.
In 1907, Mantal \cite{mantel1907} determined the extremal number for triangle for all $n\ge3$.
Tur\'an \cite{turan1941} extended Mantel's result to complete graph with any given order in 1941.

Our notations are standard, see \cite{Bondy}.
Given a graph $H$ and a set of vertices $A\subseteq V(H)$, we denote $\min\{\deg_H(x)\text{ ; }x\in A\}$ by $\delta_H(A)$.
Given a graph $H$ and a positive integer $p\ge2$, the {\it edge blow-up} of $H$, denoted by $H^{p+1}$,
is the graph obtained from $H$ by replacing each edge of $H$ by a clique of size $p+1$ where the new vertices of the cliques are all distinct.
In \cite{liu2013,ni2020} and \cite{yuan2019}, $\ex(n,H^{p+1})$ has been investgated for a large family of graphs $H$.
In \cite{wang2021}, Wang, Hou, Liu and Ma determined the extremal number when $H$ is tree satisfies some conditions and $p\ge3$.
Furthermore, the authors of \cite{wang2021} posed the following question.

\begin{question}\label{question1}
    Give $p\ge3$ and a tree $T$ such that its two coloring classes $A$ and $B$ satisfying $|A|\le|B|$,
	determine $\ex(n,T^{p+1})$ when $\delta_T(A)=1$ and $\alpha(T)>|B|$.
\end{question}
We solve this question.
First we introduce some notations.
Given two disjoint graphs $G$ and $H$, the {\it disjoint union} of $G$ and $H$, denoted by $G\cup H$,
is the graph with vertex set $V(G)\cup V(H)$ and edge set $E(G)\cup E(H)$.
We use $kG$ to denote the disjoint union of $k$ copies of $G$.
The {\it join} of $G$ and $H$, denoted by $G+H$, is the graph obtained from $G\cup H$ by adding all edges of the form $gh$, where $g\in V(G)$ and $h\in V(H)$.
Denoted by $P_n$ a path on $n$ vertices, $S_n$ a star on $n$ vertices, $C_n$ a cycle on $n$ vertices, $M_n$ a matching on $n$ vertices and
$K_{n_1,\cdots,n_p}$ the complete $p$-partite graph with the size of $i$-partite class $n_i$.
A $p$-partite Tur\'an graph on $n$ vertices, denoted by $T(n,p)$, is a $K_{n_1,\cdots,n_p}$ with $\sum_{i=1}^p n_i=n$ and $|n_i-n_j|\le1$ for $1\le i,j\le p$.
Let $H'(n,p,q)=\overline{K}_{q-1}+T(n-q+1,p)$ and $h'(n,p,q)=e(H'(n,p,q))$.
Let $e(T(n,p))=t(n,p)$.
\begin{definition}[Simonovits \cite{simonovits1974}]
    Given a family of graphs $\cL$ with $p(\cL)=p \ge 3$, let $\cM:=\cM(L)$ be a family of minimal graphs $M$ up to subgraph senses such that
    there exist a large constant $t=t(\cL)$ depending on $\cL$ such that there exists a graph $L\in\cL$ such that
	$L$ is a subgraph of $M\cup I_v + T$, where $T=T(t,p-2)$ and $I_v$ is an independent set on $v$ vertices.
    We call $\cM(\cL)$ the {\it decomposition family} of $\cL$.
\end{definition}
A {\it covering} of a graph is a set of vertices $U$ such that every edge of this graph meets at least one vertices of $U$.
An {\it independent covering} of a bipartite graph is a covering $U$ such that no two vertices of $U$ are adjacent.
The {\it covering number} $\beta(T)$ of a graph $T$ is the minimum order of a covering of $T$.
The {\it independent covering number} $q(T)$ of a bipartite graph $T$ is the minimum order of an independent covering of $T$.
The {\it independent number} $\alpha(T)$ of a graph $T$ is the maximum order of a set of vertices such that no two of which are adjacent.
For a family of graphs $\cF$ which contains at least one bipartite graph, the {\it independent covering number} of $\cF$ is defined by
$$q(\cF)=\min\{q(F)\text{ ; $F\in\cF$ and $F$ is bipartite.}\}$$

\begin{theorem}[Liu \cite{liu2013}]
	Let $p\ge3$ be an integer and $T$ be a tree.
	Let coloring classes of $T$ be $A$ and $B$, where $|A|\le|B|$.
	When $n$ is sufficiently large, we have that
\begin{itemize}
	\item if $\delta_T(A)=1$ and $\alpha(T)=|B|$,
	then $\ex(n,T^{p+1})=h(n,p,|A|)$;

	\item if $\delta_T(A)\ge2$,
	then $\ex(n,T^{p+1})=h(n,p,|A|)+1$.
\end{itemize}
    Furthermore, extremal graphs are characterized.
\end{theorem}
    Wang, Hou, Liu and Ma \cite{wang2021}  extended Liu's result to a larger family of trees very recently.
    Before stating their results, we need follow definitions. \\
    Define
\begin{equation*}
	g_1(k)=
	\begin{cases}
		k^2-\frac{3}{2}k & \text{$k$ is even;} \\
		k^2-\frac{3k-1}{2} & \text{$k$ is odd,}
	\end{cases}
	\text{ and }
	g_2(k)=
	\begin{cases}
		k^2-\frac{3}{2}k & \text{$k$ is even;} \\
		k^2-k & \text{$k$ is odd.}
	\end{cases}
\end{equation*}
\begin{theorem}[Wang, Hou, Liu and Ma \cite{wang2021}]
	Let $p\ge3$ be an integer and $T$ be a tree.
	Let coloring classes of $T$ be $A$ and $B$, where $|A|\le|B|$.
	Let $A_0=\{x\in A\text{ ; }\deg_T(x)=\delta_T(A)\}$ and $B_0=\{y\in B\text{ ; }|N(y)\cap A_0|\ge2\}$.
	Denote by $q=|A|$, $k=\delta_T(A)$ and $b+2=\delta(B_0)$.
	If $k\ge2$, then for sufficiently large $n$, we have $\ex(n,T^{p+1})=$
	\begin{equation*}
		\begin{cases}
			h(n,p,q)+g_1(k) & \text{$k$ is even;} \\
			h(n,p,q)+g_2(k) & \text{$k$ is odd and $B_0=\emptyset$;} \\
			h(n,p,q)+g_1(k) & \text{$k$ is odd and $0\le b\le q-1-\ceil{\frac{k-1}{q-1}}$}; \\
			h'(n,p,q)+g_2(k)+ \lfloor(q-1)(b-1)/2\rfloor & \text{$k$ is odd and $b\ge\max\left\{1,q-1-\ceil{\frac{k-1}{q-1}}\right\}$}. \\
			
		\end{cases}
	\end{equation*}
    Furthermore, all extremal graphs are characterized.
\end{theorem}

Now we set $\cM=\cM(T^{p+1})$ and $q=q(\cM)$.
If there exists a graph $T'\in\cM$ such that $\beta(T')\le q-1$, then
we let $\cB:=\cB(T)$ be the family of graph $T'[A_{T'}]$, where $T'\in\cM$ and $A_{T'}$ is a covering set with size at most $q-1$ of $T'$.
If $\beta(T')\ge q$ for every $T'\in\cM$, then we set $\cB=\{K_q\}$.

\begin{theorem}\label{THM:main result 1}
	Let $p\ge3$ be an integer and $T$ be a tree.
	Let coloring classes of $T$ be $A$ and $B$, where $|A|\le|B|$.
	If $\delta_T(A)=1$ and $\alpha(T)>|B|$,
	then for sufficiently large $n$, we have
	\begin{equation}
		\ex(n,T^{p+1})=h'(n,p,q)+\ex(q-1,\cB)
	\end{equation}
	where $q$ and $\cB$ are defined as above.
	Furthermore, all extremal graphs are characterized.
\end{theorem}

\noindent {\bf Remark.}	Combining with the results in \cite{liu2013} and \cite{ni2020}, the extremal number for $T^{p+1}$ is determined, where $T$ is an arbitrary tree and $p\ge3$.

\medskip

A double broom $B(\ell,s,t)$ is a tree obtained from a path $P_{\ell}$
by attaching $s$ pendant edges to one end vertex of $P_{\ell}$
and $t$ pendant edges to the other end vertex of $P_{\ell}$, where $\ell,s,t\ge2$.
The double broom $B(7,5,3)$ is as in Figure~\ref{fig:double broom}.
\begin{figure}[H]
	\begin{center}
		\begin{tikzpicture}
			\filldraw[fill=black] (0,0) circle(1.5pt);
			\filldraw[fill=black] (1,0) circle(1.5pt);
			\filldraw[fill=black] (2,0) circle(1.5pt);
			\filldraw[fill=black] (3,0) circle(1.5pt);
			\filldraw[fill=black] (-1,0) circle(1.5pt);
			\filldraw[fill=black] (-2,0) circle(1.5pt);
			\filldraw[fill=black] (-3,0) circle(1.5pt);
			\filldraw[fill=black] (4,0) circle(1.5pt);
			\filldraw[fill=black] (4,1) circle(1.5pt);
			\filldraw[fill=black] (4,-1) circle(1.5pt);
			\filldraw[fill=black] (4,0.5) circle(1.5pt);
			\filldraw[fill=black] (4,-0.5) circle(1.5pt);
			\filldraw[fill=black] (-4,0) circle(1.5pt);
			\filldraw[fill=black] (-4,1) circle(1.5pt);
			\filldraw[fill=black] (-4,-1) circle(1.5pt);
			\draw  (-4,0) -- (4,0);
			\draw  (-4,1) -- (-3,0);
			\draw  (-4,-1) -- (-3,0);
			\draw  (4,1) -- (3,0);
			\draw  (4,-1) -- (3,0);
			\draw  (4,0.5) -- (3,0);
			\draw  (4,-0.5) -- (3,0);
		\end{tikzpicture}
		\caption{Double broom $B(7,5,3)$}
		\label{fig:double broom}
	\end{center}
\end{figure}
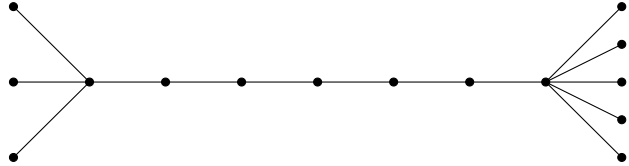

\begin{corollary}
	Let $p\ge3$ be an integer and $T=B(2k,s,t)$ be a double broom satisfying $k,s,t\ge2$.
	Then for suffciently large $n$, we have
	$$
	\ex(n,T^{p+1})=h(n,p,k+1)
	$$
	Furthermore, $H(n,p,k+1)$ is the unique extremal graph.
\end{corollary}

\begin{proof}
	It can be easily checked that $q(\cM(T^{p+1}))=k+1$ and $\beta(T')\ge k+1$ holds for every $T'\in\cM(T^{p+1})$.
	Furthermore, we have $\alpha(T)=k-1+s+t>k+\min\{s,t\}=|B|$ and $\delta_T(A)=1$.
	Therefore, the result holds by applying Theorem~\ref{THM:main result 1} with $q=k+1$ and $\cB=\{K_{q}\}$.
\end{proof}

\section{Preliminaries}
\subsection{Technical lemmas}

Given a graph $T$, a {\it vertex split} on some vertex $v\in V(T)$ is defined by replacing $v$ by an independent set of size $\deg_T(v)$
in which each vertex is adjacent to exactly one distinct vertex in $N_T(v)$.
The family of graphs that can be obtained by applying vertex split on some $U\subseteq V(T)$ is denoted by $\cH(T)$.
The following lemma can help us to determine the graphs in $\cM(T^{p+1})$.
\begin{lemma}[Liu \cite{liu2013}]\label{lem:decomposition family after splitting}
	Given $p\ge3$ and any graph with $\chi(H)\le p-1$, we have $\cM(H^{p+1})=\cH(H)$.
	In particular, a matching of size $e(H)$ is in $\cM(H^{p+1})$.
\end{lemma}

\begin{theorem}[Erd\'{o}s and Stone \cite{erdos1946}]\label{THM:turan graph}
	For all integers $p\ge1$, $N\ge1$, and every $\varepsilon>0$, there exists an integer $n_0(\varepsilon,N,p+1)$ such that every graph with $n\ge n_0$ vertices and at least $t(n,p)+\varepsilon n^2$ edges contains $T(N,p+1)$ as a subgraph.
\end{theorem}
	
\section{Proof of Theorem~\ref{THM:main result 1}}

Given a tree $T$ with coloring classes $A$ and $B$ satisfying $|A|\le|B|$, $\delta_T(A)=1$ and $\alpha(T)>|B|$.
Now we set $\cM$ be the decomposition family of $T^{p+1}$.
It follows from Lemma~\ref{lem:decomposition family after splitting} that $\cM=\cH(T)$.
Furthermore, $\cM$ contains a matching of size $t$, where $t=e(T)$.

Let $\mathfrak{U_n}$ be the family of graphs obtained from $H'(n,p,q)$ by embedding a copy of $Q\in\EX(q-1,\cB)$ in $\ol{K}_{q-1}$ in $H'(n,p,q)$.
The definition of $\cB$ implies every $H_n\in\mathfrak{U_n}$ is $T^{p+1}$-free, and hence we have
\begin{equation}\label{eq:ex lower bound}
	\ex(n,T^{p+1})\ge h'(n,p,q)+\ex(q-1,\cB)
\end{equation}
Let $\phi(n):=\ex(n,T^{p+1})-h'(n,p,q)-\ex(q-1,\cB)$ and $K=\max\{\phi(n):n\le n_0\}$, where $n_0$ is a large constant depending on $p$ and $T$.
Clearly,  $\phi(n)$ is a non-negative integer.
For the upper bound, we will show that if $n>n_0$ and $\phi(n)>0$, then there exists an $n_4$ depending on $p$ and $T$ such that $\phi(n)<\phi(n-n_4p)$.
This would imply  that if $n= n_0+mn_4p$, then $\phi(n)<K-m$, and hence the theorem holds for $n\ge n_0+Kn_4p$.

Let $n_1$ be a sufficiently large constant.
Let $n_0=n_0(\varepsilon,n_1p,p)$ be the constant from Theorem~\ref{THM:turan graph}, where $\varepsilon=1/(2p(p-1))$.
Let $L_n$ be a $T^{p+1}$-free graph with $\ex(n,T^{p+1})$ edges, where $n\ge n_0$.
Equation~\ref{eq:ex lower bound} and Theorem~\ref{THM:turan graph} imply that
$L_n$ contains a $T=T(n_1p,p)$ with partite class $\wt{B}_1^0,\cdots,\wt{B}_p^0$ as a subgraph.
Note that $M_{2t}\in\cM$.
It follows from the definition of decomposition family and the fact that $L_n$ is $T^{p+1}$-free that $\nu(L_n[\wt{B}_i^0])\le t$ for $i\in\brac{p}$.
Let the maximum matching in $L_n[\wt{B}_i^0]$ be $\{x_1y_1,\cdots,x_{t_i},y_{t_i}\}$ with $t_i\le t$.
Let $B_i^0=\wt{B}_i^0 \setminus\{x_1y_1,\cdots,x_{t_i},y_{t_i}\}$.
By the definition of $B_i^0$, there is no edge in $L_n[B_i^0]$.
Hence there is an induced subgraph $T_0=T(n_2p,p)$ of $L_n$ with partite class $B_1^0,\cdots,B_p^0$ obtained by deleting $2t$ vertices from each $\wt{B}_i$, where $n_2=n_1-2t$.

Let $c<1/(1+t)$ be a sufficiently small constant.
If there exists a vertex $x_1\in L_n-T_0$ such that $x_1$ is adjacent to at least $c^2n_2$ vertices of each partite class of $T_0$,
then $T_0$ contains a $T_1=T(c^2n_2p,p)$ such that each vertex of which is joint to $x_1$.
Generally, if there exists a vertex $x_i\in L_n-T_{i-1}-\{x_1,\cdots,x_{i-1}\}$ such that $u$ is adjacent to at least $c^{2i}n_2$ vertices of each partite class of $T_{i-1}$,
then $T_{i-1}$ contains $T_i=T(c^{2i}n_2p,p)$ such that each vertex of which is joint to $x_1,\cdots,x_i$.
Thus we can define a sequence of graphs recursively.
However, it follows from the definition of $L_n$ and $q$ that the above process stops at last after $T_{q-1}$.
Suppose to the contrary, let $V(T_q)=B_1^q\cup\cdots\cup B_p^q$.
Note that the graph induced by $B_1^q\cup\{x_1,\cdots,x_q\}$ contains some element of $\cM$ by the definition of $q$.
Then $L_n$ contains a copy of $T^{p+1}$ by the definition of decomposition family, a contradiction.

Now suppose that the above process ends with $T_s$ with $s\le q-1$.
Let $E=\{x_1,\cdots,x_s\}$ and the partite class of $T_s$ be $B_1^s,\cdots,B_p^s$.
Denote $|B_1^s|$ by $n_3$ for convenience.
We can partition the remaining vertices into following set: Let $x\in V(L_n)\setminus(T_s\cup E)$.
If there exists an $i\in\brac{p}$ such that $x$ is adjacent to less than $c^2n_3$ vertices of $B_i^s$ and is adjacent to at least $(1-c)n_3$ vertices of $B_j^s$ for all $j\ne i$, then let $x\in C_i$.
If there exists an $i\in\brac{p}$ such that $x$ is adjacent to less than $c^2n_3$ vertices of $B_i^s$ and is adjacent to less than $(1-c)n_3$ vertices of $B_j^s$ for some $j\ne i$, then let $x\in D$.
It follows from the definition of $T_s$ that $C_1\cup\cdots\cup C_p\cup D$ is a partition of $V(L_n)\setminus(T_s\cup E)$.
Note that for a $S\subset B_i^s\cup C_i$ with $|S|\le2t$, the common neighbourhoods of $S$ in $B_j^s$ is at least $(1-2tc)n_3\ge n_3/2$, where $j\ne i$.
It follows from the definition of decomposition family and Lemma~\ref{lem:decomposition family after splitting} that $\nu(L_n[B_i^s\cup C_i])\le t$.
Now consider the edges joining $B_i^s$ and $C_i$ and select a maximum matching, say $y_1z_1,\cdots,y_{t_i}z_{t_i}$ with $y_{i'}\in B_i^s$, $z_{i'}\in C_i$ and $1\le i'\le t_i\le t$.
Let $X_i=\cup_{i'=1}^{t_i} (N_{L_n}(z_{i'})\cap B_i^s)$. Then $|X_i|\le tc^2n_3$ by the definition of $C_i$.
Let $C_i'=C_i\cup X_i$ and $B_i'=B_i^s\setminus X_i$, then $L_n[B_i'\cup C_i']$ contains no edge by the maximality of $y_1z_1,\cdots,y_{t_i}z_{t_i}$.
Hence it is possible to move $tc^2n_3$ vertices from $B_i^s$ to $C_i$ to obtain $B_i'$ and $C_i'$ such that $B_i'\subset B_i^s$ and $C_i\subset C_i'$.
Let $n_4=(1-tc^2)n_3=|B_i'|$, $T_s'=T(n_4p,p)$ and $\wh{L}=L_n-T_s'$.
Then $T_s'$ is an induced subgraph of $L_n$ and the vertices of $\wh{L}$ can be partitioned into $p+2$ sets $C_1',\cdots,C_p',D$ and $E$ such that
\begin{itemize}
\item every $x\in E$ is adjacent to each vertex of $T_s'$ and $|E|=s$,
\item every $x\in C_i'$ is adjacent to no vertex of $B_i'$ and
is adjacent to at least $(1-c-tc^2)n_3$ vertices of $B_j'$ for all $j\ne i$.
\item every $x\in D$ is adjacent to at most $c^2n_3$ vertices of $B_i'$ and
is adjacent to at most $(1-c)n_3$ vertices of $B_j'$ for some $i,j\in \brac{p}$ with $i\ne j$.
\end{itemize}
\noindent Let the number of edges joining $T_s'$ and $\wh{L}$ in graph $L_n$ denoted by $e_L$.
Then we have
\begin{equation*}
	e(L_n)=e(\wh{L})+e_L+e(T_s')
\end{equation*}
Let $H_n\in \mathfrak{U}_n$ and $T_s''$ be an induced copy of $T(n_4p,p)$ in $H_n$.
Let $H_{n-n_4p}=H_n-T_s''$ and $e_H$ be the number of edges joining $T_s''$ and $H_{n-n_4p}$ in graph $H_n$.
Then
\begin{equation*}
	e(H_n)=e(H_{n-n_4p})+e_H+e(T_s'')
\end{equation*}
Since $\wh{L}$ contains no copy of $T^{p+1}$, we have $e(\wh{L})\le e(L_{n-n_4p})$, where $L_{n-n_4p}\in\EX(n-n_4p,T^{p+1})$.
Obviously we have $e(T_s')=e(T_s'')$.
Simple calculation show that
\begin{equation}\label{equ:e_H}
	\begin{aligned}
		e_H
		&=(q-1)n_4p+(n-n_4p-q+1)n_4(p-1)\\
		&=(q-1)n_4+(n-n_4p)n_4(p-1)\\
	\end{aligned}
\end{equation}
It follows from the definition of $C_i'$, $D$ and $E$ that
\begin{equation}\label{equ:e_L}
	\begin{aligned}
		e_L
		&\le sn_4p+(n-n_4p-s-|D|)n_4(p-1)+|D|((p-2)n_4+(1-c+c^2)n_3) \\
		&=sn_4+(n-n_4p)n_4(p-1)-|D|(n_4-(1-c+c^2)n_3) \\
		&\le(q-1)n_4+(n-n_4p)n_4(p-1)-|D|n_3(c-(t+1)c^2) \\
		&=e_H-|D|n_3(c-(t+1)c^2)
	\end{aligned}
\end{equation}
Hence we have
\begin{equation*}
	\begin{aligned}
		\phi(n)
		&=e(L_n)-e(H_n) \\
		&\le e(L_{n-n_4p})-e(H_{n-n_4p})+e_L-e_H \\
		&=\phi(n-n_4p)+e_L-e_H
	\end{aligned}
\end{equation*}
If $e_L-e_H<0$, then we have $\phi(n)<\phi(n-n_4p)$, where $n_4\le n_2$.
Hence we suppose that $e_L-e_H\ge 0$.
Combined with Equation~\ref{equ:e_H} and~\ref{equ:e_L}, we conclude that $e_L=e_H$. (Note that $c<1/(1+t)$ is sufficiently small.)
Note that $e_L=e_H$ holds if and only if $|D|=0$, $s=q-1$ and $C_i'$ is complete to $B_j'$ for $i\in\brac{p}$ and $j\ne i$.

If $e(L_n[E])$ contains some copy of $B'\in\cB$, then $L_n$ contains a copy of $T^{p+1}$ by the definition of $\cB$.
Hence we conclude that $L_n[E]$ is $\cB$-free and $e(L_n[E])\le\ex(q-1,\cB)$.
The rest of the proof will be divided into two cases.

\medskip

\noindent{\bf Case 1. $ q=q(T)$.}

\medskip

In this case, note that $T$ is a tree.
Clearly, $T$ admits a unique proper $2$-coloring and hence $q(T)=|A|$ holds.
Note that $\delta_T(A)=\min\{\deg_T(x)\text{ ; }x\in A\}=1$ by assumptions.
Hence there exists a vertex $u\in A$ such that $N_T(u)=\{v\}$.
Since $|A|=q(T)=q$, we can find a copy of $(T-\{u\})^{p+1}$ using vertices in $E\cup B_1'\cup\cdots\cup B_p'$ in $L_n$
Let $\phi$ be an embedding from $(T-\{u\})^{p+1}$ to $L_n$ such that
$\image \psi \subseteq E\cup B_1'\cup\cdots\cup B_p'$ and $\psi(A\setminus\{u\})=E$.

Now we will show that $B_i'\cup C_i'$ is an independent set of $L_n$ for each $i$.
It suffices to show $C_i'$ is an independent set for each $i$ since there is no edge incident with $B_i'$.
We assume that $\psi(v)\in B_\ell$, where $\ell\ne i$.
\footnote{This is possible since $E$ is joint to every vertex in $B_1'\cup\cdots\cup B_p'$.}
In fact, if there is an edge $u'u''$ in $L_n[C_i']$,
then we can choose $u_{j}$ in $B_j'$ such that $u_j\notin \image \psi$ for $j\in\brac{p}\setminus\{i,\ell\}$.
It can be seen immediately from the definition of $B_i'$, $C_i'$ and $E$ that
$u'$, $u''$ and $\psi(v)$ togerther with all $u_j$ forms a copy of $K_{p+1}$.
Furthermore, it can be verified that the mapping constructed above is an embedding from $T^{p+1}$ to $L_n$.
This completes the proof for this case.

\medskip

\noindent{\bf Case 2. $ q<\sigma(T)$.}

\medskip

Let $F\in\cM$ such that $q(F)=q$. Let $A_F$ and $B_F$ be coloring classes of $F$ such that $q(F)=|A_F|$.
Now we show that $\min\{\deg_F(x)\text{ ; }x\in A_F\}=1$.
It follows from the definition of decomposition family that $\min\{\deg_F(x)\text{ ; }x\in A_F\}\ge1$.

If $A_F$ contains a vertex $u$ which is obtained by splitting a vertex in $T$, then the result follows since $\deg_F(u)=1$.
Now we assume that every $u\in A_F$ is not a vertex obtained by splitting a vertex in $T$.
Then we have $u\in V(T)$ for every $u\in A_F$.
By lemma~\ref{lem:decomposition family after splitting}, we may assume $F$ is obtained by splitting $X\subseteq V(T)$.
It is easy to see that $X\cap A_F=\emptyset$.
Otherwise, we can find a vertex obtained by splitting a vertex in $T$, a contradiction.
Let the vertices obtained by splitting $X$ in $T$ be $Y$ and $Z=B_F\setminus X$.
It is clear that $V(T)$ is the disjoint union of $A_F$, $X$ and $Z$.
Furthermore, $V(F)$ is the disjoint union of $A_F$, $Y$ and $Z$.
Note that we have $E_T(A_F,Z)=E_F(A_F,Z)$ by the definition of $Z$.
It follows from the definition that $Y\cup Z$ is an independent set of $F$.
Note that $\delta(F)\ge1$ since $F$ is obtained by splitting $X\subseteq V(T)$.
Hence every $y\in Y$ is adjacent to some $v\in A_F$ in graph $F$.
Therefore, $A_F$ is a independent covering of $T$.
Then we have $q(T)\le|A_F|=q<q(T)$, a contradiction.

Hence we have $\min\{\deg_F(x)\text{ ; }x\in A_F\}=1$.
Then similar arguments in Case 1 show that $B_i'\cup C_i'$ is an independent set of $L_n$.

Note that $B_i'\cup C_i'$ is an independent set of $L_n$ for each $i$ in both cases, then we have
\begin{equation}
	\begin{aligned}
		e(L_n)
		&\le L_n[E]+e_{L_n}(E,V(L_n)\setminus E)+\sum_{1\le i<j\le p}e_{L_n}(B_i'\cup C_i',B_j'\cup C_j' ) \\
		&\le \ex(q-1,\cB)+(q-1)(n-q+1)+\sum_{1\le i<j\le p}|B_i'\cup C_i'||B_j'\cup C_j'| \\
		&\le \ex(q-1,\cB)+(q-1)(n-q+1)+t(n-q+1,p) \\
		&= \ex(q-1,\cB)+h'(n,p,q)
	\end{aligned}
\end{equation}
which contradicts the fact that $\phi(n)>0$. The theorem follows.

\end{document}